\documentclass[11pt]{article}

\usepackage[colorlinks=true,citecolor=black,linkcolor=black,urlcolor=blue]{hyperref}
\usepackage[english]{babel}
\usepackage[utf8]{inputenc}
\usepackage{latexsym}
\usepackage{booktabs}
\usepackage{amsfonts}
\usepackage[pdftex]{graphicx}
\usepackage{subfig}
\usepackage{caption}
\usepackage{textcomp}
\usepackage{amsmath,amssymb,amsthm}
\usepackage{tikz,color,graphicx}
\usepackage{pgfplots}
\usepackage{floatflt,epsfig}
\usepackage{bbm}
\usepackage{blkarray}
\usepackage{mathtools}
\usepackage{tikz}
\usetikzlibrary{positioning,chains,fit,shapes,calc}
\usetikzlibrary{arrows,intersections}

\usepackage{color}
\usepackage[colorinlistoftodos]{todonotes}

\theoremstyle{plain}
\newtheorem{theorem}{Theorem}[section]
\theoremstyle{definition}
\newtheorem{definition}[theorem]{Definition}
\theoremstyle{definition}

\theoremstyle{definition}

\theoremstyle{definition}

\theoremstyle{plain}
\newtheorem{corollary}[theorem]{Corollary}
\theoremstyle{plain}
\newtheorem{proposition}[theorem]{Proposition}
\theoremstyle{plain}
\newtheorem{lemma}[theorem]{Lemma}
\theoremstyle{plain}

\theoremstyle{definition}
\newtheorem{remark}[theorem]{Remark}

\numberwithin{equation}{section}
\numberwithin{theorem}{section}

\newcommand{\sss}{\scriptscriptstyle}
\def\cN{\mathcal{N}}
\def\E{\mathbb{E}}

\def\vep{\varepsilon}
\def\ep{\varepsilon}

\def\Nlos{N_{\mathrm{los}}}
\def\P{\mathbb{P}}
\newcommand{\ch}[1]{\textcolor{black}{#1}}
\newcommand{\eqn}[1]{\begin{equation}#1\end{equation}}
\newcommand{\eqan}[1]{\begin{align}#1\end{align}}

\DeclareMathOperator{\expec}{{\mathbb{E}}}

\newcommand {\convp}{\stackrel{\sss {\mathbb P}}{\longrightarrow}}

\begin{document}

\pagestyle{plain}
\pagenumbering{arabic}

\title{The winner takes it all but one}


\author{\renewcommand{\thefootnote}{\arabic{footnote}}
\renewcommand{\thefootnote}{\arabic{footnote}}
Maria Deijfen%
\footnotemark[1]\hspace{0.5cm}
\renewcommand{\thefootnote}{\arabic{footnote}}
Remco van der Hofstad%
\footnotemark[2]\hspace{0.5cm}
\renewcommand{\thefootnote}{\arabic{footnote}}
Matteo Sfragara\,%
\footnotemark[1]
}

\footnotetext[1]{%
Department of Mathematics, Stockholm University, Sweden.
}
\footnotetext[2]{%
Department of Mathematics and Computer Science, Eindhoven University of Technology, The Netherlands.
}

\date{April 2022}
\maketitle

\begin{abstract}

\noindent We study competing first passage percolation on graphs generated by the configuration model with \ch{infinite-mean degrees}. Initially, two uniformly chosen vertices are infected with type 1 and type 2 infection, respectively, and the infection then spreads via nearest neighbors in the graph. The time it takes for the type 1 (resp.\ 2) infection to traverse an edge $e$ is given by a random variable $X_1(e)$ (resp.\ $X_2(e)$) and, if the vertex at the other end of the edge is still uninfected, it then becomes type 1 (resp.\ 2) infected and immune to the other type. Assuming that the degrees follow a power-law distribution with exponent $\tau \in (1,2)$, we show that, with high probability as the number of vertices tends to infinity, one of the infection types occupies all vertices except for the starting point of the other type. Moreover, both infections have a positive probability of winning regardless of the passage times distribution. The result is also shown to hold for the erased configuration model, where self-loops are erased and multiple edges are merged, and when the degrees are conditioned to be smaller than $n^\alpha$ for some $\alpha > 0$.
\end{abstract}


\section{Introduction and main results}

First passage percolation (FPP) was introduced in \cite{HW65} as a model for the flow of fluid through random media and has evolved into one of the fundamental models of random growth. The basic model for FPP on a graph is defined by assigning i.i.d.\ non-negative random weights to the edges, referred to as passage times and interpreted as the times or costs of traversing the edges. This induces a random metric on the vertex set where the distance between two vertices is the minimal cost-sum among all nearest-neighbor paths connecting the two vertices. Of primary interest is the asymptotic behavior of distances, balls and geodesics (time-minimizing paths) in the first-passage metric. The classical example is when the underlying structure is taken to be the $\mathbb{Z}^d$-lattice; see \cite{50years} for a recent survey of results in this setting. The case with exponential passage times has received particular attention and is referred to as the Richardson model.

In \cite{HP98}, the Richardson model (on $\mathbb{Z}^d$) was extended to a two-type version that describes a {\em competition} between two infection types that evolve simultaneously using passage times with (potentially) different intensities. The event that both infection types occupy infinite parts of the lattice is referred to as {\em infinite coexistence,} and it is conjectured that this has positive probability if and only if the infections have the same intensity. The if-direction was proved in full generality independently in \cite{GM05} and \cite{H05}. The only-if-direction remains to be fully resolved, but partial results can be found in \cite{HP00}. We refer to \cite{DH08} for a survey and further references.

The past few years have witnessed an explosion in the amount of empirical data on networks, showing that many networks exhibit similar properties. This has incited the formulation of a large number of network models aiming at capturing and explaining these properties. One property that is observed in many real-world networks is that they display asymptotic power-law degree distributions, that is, the number of vertices with degree $k$ decays asymptotically for large networks as well as large $k$ as $k^{-\tau}$ for some exponent $\tau>1$, which for a variety of empirical networks has been observed to range from just above 1 to a bit above 3; see \cite[Table II]{AB02}. See also \cite{VvdHvdHK19} for more recent estimates of power-law exponents, and \cite{BC19} for criticism on the prevailing suggestion that power laws are omnipresent. The regime $\tau >3$ corresponds to finite variance, $\tau \in (2,3)$ to finite mean but infinite variance, and $\tau \in (1,2)$ to infinite mean. The standard model for generating graphs with a prescribed degree distribution is known as the {\em configuration model} and is constructed by independently assigning half-edges to the vertices according to the desired distribution and then pairing the half-edges randomly. Its structure is well understood in all three power-law regimes mentioned above; see e.g.\ \cite{vdEvdHHZ05,vdHHVM05,vdHHZ05,vdHHZ07}. FPP on the configuration model with $\tau>2$ and exponential passage times has been analyzed in \cite{BvdHH10a} and the results for $\tau>3$ were extended to all continuous passage-time distributions in \cite{BvdHH17}. For $\tau \in (2,3)$, FPP was studied in more detail in \cite{BvdHK17, BvdHK19}, and for $\tau \in (1,2)$ in \cite{BvdHH10b}.

In competing FPP on the configuration model, two infection types compete to invade the vertices in the graph. Each edge is equipped with two independent non-negative random weights from two potentially different distributions, indicating the passage time for type 1 and 2, respectively. When a vertex is type $i$ infected, the passage times on the incident edges are activated. Then, a given neighbor that is uninfected when the type $i$ passage time on the connecting edge has passed becomes type $i$ infected at that time. Infected vertices become immune to the other type and stay infected forever. The growth is typically started from two uniformly chosen vertices. The model was analyzed for exponential passage times and constant degrees (leading to random regular graphs) in \cite{ADMP17}, where it was shown that the strongest type occupies all but a vanishing fraction of the vertices when the intensities are different, while both types occupy positive fractions of the vertices when they are equal (the analysis in \cite{ADMP17} included also more general initial sets leading to modified results). For exponential passage times and degree exponent $\tau>3$, it was shown in \cite{ADJ19} that the behavior is the same as for constant degrees. The case with exponential passage times and $\tau\in(2,3)$ was considered in \cite{DvdH16}, where it was shown that one of the types then occupies all but a finite number of the vertices. Furthermore, both types have a positive probability of winning regardless of the intensities of the infections. The model has also been analyzed for constant passage times in the regime $\tau\in(2,3)$ in \cite{BvdHK15} and \cite{vdHK15}. When the types have different (constant) passage times, the faster type occupies all but a vanishing fraction of the vertices \cite{BvdHK15}, while in the symmetric case coexistence can occur depending on the choice of the two starting vertices \cite{vdHK15}.

In this paper we analyze competing FPP on graphs generated by the configuration model with degree exponent $\tau\in(1,2)$ for a large class of passage-time distributions. Values of $\tau \in (1, 2)$ have been observed in social networks, such as e-mail networks and collaboration networks, in technological networks, such as the link structure of the World Wide Web and networks of dependencies between software packages, and in ecological networks; see \cite[Table II]{AB02} and \cite[Table II]{N03}. We show that, with high probability as the number of vertices grows to infinity, one of the infection types occupies all vertices except for the starting point of the other type, \ch{leading to the most extreme `winner-takes-it-all' phenomenon possible}. Moreover, both types have a positive probability of winning, regardless of the passage times distributions. For $\tau\in(1,2)$, the graph has a degenerate structure where all vertices are connected to a small number of giant-degree vertices with degrees comparable to the total degree in the graph. The competition is essentially won by the type that first makes it from its starting point to one of its (giant-degree) neighbors. After this happens, the infection type quickly invades all other giant-degree vertices, thereby preventing the other type from making any progress at all. The behavior is explosive and the outcome of the competition is determined in finite time.

\subsection{Definition of the model}
\label{sec:model}
The configuration model takes $n$ vertices and a probability distribution with support on non-negative integers as input. Let $D$ be a random variable drawn from the given probability distribution and let $D_1, D_2, \dots, D_n$ be i.i.d.\ copies. These represent the degrees of the vertices and we write $L_n = \sum_{i=1}^n D_i$ for the total degree. To construct the graph, each vertex $i\in \{1, \ldots, n\}=[n]$ is first assigned $D_i$ half-edges. The half-edges are then iteratively paired to form edges. Specifically, at each step we pick two half-edges uniformly at random from the set of half-edges that have not been paired yet, and connect them into an edge. If $L_n$ is odd, so that only one half-edge remains in the last step, we add one extra half-edge at a uniformly chosen vertex. To avoid trivial complications in the formulation of our results, we throughout assume that $D_i\geq 1$, so that there are no isolated vertices in the graph. 

The probability mass function of the degree distribution is denoted by
$$
\mathbb{P}(D = j) = f_j, \qquad j = 1,2,\dots,
$$
and the distribution function is given by
$$
F(x) = \mathbb{P} (D \leq x) = \sum_{j=1}^{\lfloor x \rfloor} f_j,
$$
where $\lfloor x \rfloor$ indicates the largest integer smaller than or equal to $x$. Our main assumption is that
\begin{equation}
\label{svfunction}
1 - F(x) = \ell(x)x^{-(\tau-1)}\quad\mbox{for some }\tau \in (1,2),
\end{equation}
where $\ell$ is a slowly-varying function, that is, the degrees obey a power law with infinite mean. The notation $\ell(n)$ will throughout refer to a slowly varying function that might differ at different occurrences.

To define the competition process, we equip each edge $e$ in the graph with two independent weights  $X_1(e)$ and $X_2(e)$ representing the passage time through the edge for type 1 and 2, respectively. Here $\ch{(X_1(e))}_e$ and $\ch{(X_2(e))}_e$ are assumed to be i.i.d.\ but with potentially different distributions. At time 0, vertex 1 is infected with type 1 and vertex 2 is infected with type 2 while all other vertices are uninfected. Note that, since vertices are exchangeable in the configuration model, this is equivalent to starting from two vertices chosen uniformly at random. The infections then spread in the graph via nearest neighbors: the time it takes for the type 1 (resp.\ 2) infection to traverse an edge $e$ and reach the vertex $v$ at the other end of the edge is given by $X_1(e)$ (resp. $X_2(e)$). If $v$ is still uninfected at that time, \ch{then} it becomes type 1 (resp.\ 2) infected. An infected vertex remains infected with the same type forever and becomes immune to the other type.

For a random variable $Y$, let \ch{supp}$(Y)$ denote the support of the distribution of $Y$. We work with a fairly large class of passage time distributions, assuming only that
\begin{equation}
\label{passtime_assumption}
\begin{array}{l}
X_1(e) \mbox{ and } X_2(e) \mbox{ are continuous random variables}\\
\mbox{with inf~} \mbox{supp}(X_1(e))=\mbox{inf~}\mbox{supp}(X_2(e))=0.
\end{array}
\end{equation}

\subsection{Results and heuristics}

It is well-known that, when $\tau \in (1,2)$ and $D_i\geq 1$, the graph generated by the configuration model contains a giant component that comprises almost all vertices, that is, the asymptotic fraction of vertices in the giant component converges to 1; see \cite{ACL01,MR95}. This means that almost all vertices will eventually be infected in the competition process described above. Let $N_i(n)$ denote the number of vertices ultimately infected by type $i$ ($i=1,2$) and define $\Nlos(n)=\min\{N_1(n),N_2(n)\}$. Also write $Z_i$ for the minimal passage time from the initial type $i$ vertex to one of its neighbors, so that, assuming that $i$ does not have any self-loops, $Z_i \stackrel{d}{=} \min\{X_i(j)\colon j=1, \dots, D\}$, where $(X_i(j))_{j\geq 1}$ are i.i.d.\ random variables from the passage time distribution. Our main result states that one of the types \ch{overtakes} the other by occupying all vertices except for the starting point of the other type. Furthermore, the winning type is the one that makes the first move in the process by infecting one of its neighbors.

\begin{theorem}[The winner takes it all but one]
\label{thm:o}
Consider competing FPP on the configuration model satisfying \eqref{svfunction} and \eqref{passtime_assumption}. Then $\lim_{n\to\infty} \P(\Nlos(n)=1)=1$.
Furthermore,
$$
\lim_{n\to\infty} \P(N_1(n)=1\mid Z_1>Z_2)=\lim_{n\to\infty} \P(N_2(n)\ch{=1}\mid Z_1<Z_2)=1.
$$
\end{theorem}

Note that it follows from \eqref{passtime_assumption} that the event $\{Z_1>Z_2\}$ has a non-trivial probability, implying that both types have a positive probability of winning the whole graph except for the other starting location.

To heuristically explain the result, note that, for i.i.d.\ power-law random variables with exponent $\tau \in (1,2)$, the sum $L_n$ is of the order $n^{1/(\tau-1)\ch{+o(1)}}$ \ch{when \eqref{svfunction} holds}, which is also the scaling of the maximum degree. In terms of our configuration graph, this means that the bulk of the contribution to the total degree comes from a {\em finite} number of vertices with degrees of the same order as the total degree. \ch{We refer to these} as giant-degree vertices or simply {\em giants}. A basic fact for the configuration model is that the number of connections between two sets of half-edges of sizes $a$ and $b$ is of the order $ab/L_n$. This implies that, since the giants have degree of order $n^{1/(\tau-1)\ch{+o(1)}}$, they are all linked to each other, thus forming a tightly connected complete graph, with the number of multiple edges between two giants being of the order $n^{1/(\tau-1)\ch{+o(1)}}$. Furthermore, with high probability as $n \to \infty$, all other vertices are connected {\em only} to giants. We refer to \cite{vdEvdHHZ05} for a more detailed description of the structure of the graph.

Theorem~\ref{thm:o} is now explained in that all neighbors of the initially infected vertices are tightly connected giants, implying that the type that makes the first move by occupying one of these giants will quickly invade all other giants as well, thereby preventing the other type from making any progress at all. Indeed, the passage time between two giants is the minimum of $n^{1/(\tau-1)\ch{+o(1)}}$ (the number of multiple edges) i.i.d.\ edge passage times, which converges to 0 in probability under the assumption \eqref{passtime_assumption}. The spread between giants is hence \ch{extremely} fast, since there are many edges between giants.

Next, we investigate a more general scenario in which the competition starts from multiple vertices chosen uniformly at random. \ch{In our main result,} for a process started from $k_i$ type $i$ vertices, \ch{we denote} the minimum passage time from the initial type $i$ vertices to the set of neighbors of these vertices by $Z_{i,k_i} \stackrel{d}{=} \min\{X_i(j)\colon j=1, \dots, \sum_{l=1}^{k_i} D_l\}$.

\begin{corollary}[Multiple starting points]
\label{thm:multiple}
Consider competing FPP on the configuration model satisfying \eqref{svfunction} and \eqref{passtime_assumption}, and starting with $k_i$ type $i$ vertices chosen uniformly at random, where $k_i\geq 1$ is fixed ($i=1,2$). The number $\Nlos(n)$ converges in distribution to a random variable $W$ with $\P(W=k_1)=1-\P(W=k_2)=\P(Z_{1,k_1}>Z_{2,k_2})$. More specifically,
$$
\lim_{n\to\infty}\P(N_1(n)=k_1\mid Z_{1,k_1}>Z_{2,k_2})= \lim_{n\to\infty}\P(N_2(n)=k_2 \mid Z_{1,k_1}<Z_{2,k_2})=1.
$$
If $k_1$ is fixed and $k_2=k_2(n)\to\infty$ as $n\to\infty$, the conclusions remain valid with $W\equiv k_1$ and $\lim_{n\to\infty}\P(N_1(n)=k_1)=1$. Similarly, if $k_2$ is fixed and $k_1=k_1(n)\to\infty$, then $W\equiv k_2$ and $\lim_{n\to\infty}\P(N_2(n)=k_2)=1$.
\end{corollary}

The result is again explained by the fact that the type that first reaches a neighbor of its initial set will soon thereafter occupy all giants in the graph, thereby preventing the other type from growing beyond its initial set. If both $k_1$ and $k_2$ are fixed, then $\P(Z_{1,k_1}>Z_{2,k_2})\in(0,1)$ and both types have a positive probability of winning. If $k_1$ is fixed while $k_2$ grows with $n$, then $\P(Z_{1,k_1}>Z_{2,k_2})\to 0$, implying that the type 2 infection \ch{wins} with high probability.

\paragraph{\ch{The conditioned model.}}
The maximum degree in our configuration graph is of the order $n^{1/(\tau-1)\ch{+o(1)}}$ with $\tau\in(1,2)$. In some situations this type of very large degrees are artificial and we may want prevent this while keeping the same form of the degree distribution. We therefore extend our results to the case when the degrees are conditioned to be smaller than $n^{\alpha}$ for some $\alpha > 0$. Specifically we let the degrees be i.i.d\ copies of a random variable $D(n)$, with probability mass function
\begin{equation}
\label{cond_degrees}
\mathbb{P}(D(n) = j) =\frac{f_j}{\mathbb{P}(D \leq n^{\alpha})}\quad\mbox{for }j=1,2,\ldots\mbox{ and }\alpha>0.
\end{equation}

\begin{theorem}[Conditioned model]
\label{thm:c}
Under the assumptions \eqref{svfunction} and \eqref{passtime_assumption}, the conclusions of Theorem~\ref{thm:o} and Corollary~\ref{thm:multiple} are valid also for competing FPP on the configuration model with degree distribution \eqref{cond_degrees} for any $\alpha\neq 1/(\tau +k)$, $k\in\mathbb{N}$.
\end{theorem}

When $\alpha>1/(\tau-1)$ the conditioning has no effect on the graph, so the interesting regime is $\alpha<1/(\tau-1)$. The maximum degree in the graph is then $n^\alpha$ and the total degree $L_n=\sum_{i\in[n]} D_i(n)$ is of the order $n^{1+\alpha(2-\tau)}$. This means that all vertices are still connected only to the vertices of maximal degree $n^\alpha$, which now play the role of the giants. In contrast to the unconditioned case, the number of giants grows to infinity with $n$, indicating that the time from when one giant is infected until all giants are infected may not vanish. We show however that the time until the infection reaches all (giant) neighbors of the type that failed to make the first move does vanish. For $\alpha>1/\tau$, this follows from the fact that the giants still form a complete graph with a large number of multiple edges between them. For $\alpha<1/\tau$, the giants no longer form a complete graph, but in \cite{vdEvdHHZ05} it is shown that, for $\alpha\in(1/(\tau+k),1/(\tau+k-1))$, the graph distance between two giants is at most $k+1$. This observation can be used to show that any two giants are with high probability connected by a large number of {\em disjoint paths} of bounded length. Assuming \eqref{passtime_assumption}, this gives the desired conclusion, since the passage times of disjoint paths are independent and the sum of at most $k+1$ i.i.d.\ passage times still has 0 as the infimum of its support.

\paragraph{\ch{The erased model.}} The configuration model allows for self-loops and multiple edges between vertices. Indeed, for $\tau\in(1,2)$, these structures \ch{are} abundant and the occurrence of multiple edges is one of the explanations to the behavior of the competition process. In some situations, however, self-loops and multiple edges are not desirable. One option is then to first generate the graph and then delete all self-loops and merge all multiple edges. This is know as the {\em erased} configuration model.

The topology of a graph generated by the erased configuration model has been studied in \cite{BvdHH10b}. As a result of the erasure, there are no longer multiple edges between vertices, but the set of neighbors of a given vertex remains the same. In \cite{vdEvdHHZ05} it is shown that the number of joint neighbors of two given giants, with degree of the order $n^{1/(\tau-1)}$ before erasure, is of the order $n$, so that two giants are hence connected by a large number of disjoint two-step paths. Hence \eqref{passtime_assumption} implies that the passage time between two giants is vanishing also in the erased model, giving rise to the same behavior of the competition process:

\begin{theorem}[Erased model]
\label{thm:e}
The conclusions of Theorem~\ref{thm:o} and Corollary~\ref{thm:multiple} are valid also for competing FPP on the erased configuration model satisfying \eqref{svfunction} and \eqref{passtime_assumption}.
\end{theorem}

The rest of the paper is organized so that Section 2 contains the proofs, followed by some suggestions of further work in Section 3.


\section{Proofs}

This section contains the proofs. We first give a more formal definition of giant-degree vertices and then show a key proposition stating that the passage time between two giants is vanishing in all three instances of the model (the original one, the conditioned model and the erased model). With this at hand, all results follow without much further effort. We consider the case with two initial points and then briefly describe how the arguments can be generalized to larger initial sets at the end of the section.

We start by defining the giant-degree vertices as those with a degree of the same order as the maximal degree in the graph. To this end, let $\{u_n\}$ be a sequence satisfying
$$
1-F(u_n)=(1+o(1))/n.
$$
It follows from standard extreme value theory that $u_n$ is the scaling of the total degree $L_n$ as well as the maximal degrees in the graph; see e.g.\ \cite[Lemma 2.1]{vdEvdHHZ05} where this is formulated in the context of the configuration model. Furthermore, it follows from \eqref{svfunction} that \ch{there exists a slowly varying function $n\mapsto l(n)$ such that}
\begin{equation}\label{u_n}
u_n=l(n)n^{1/(\tau-1)}.
\end{equation}

\begin{definition}[Giant-degree vertices]
Fix a sequence $(\vep_n)$ such that $\vep_n \ch{\searrow} 0$ arbitrarily slowly. The set of giant-degree vertices (or giants) is given by $\mathcal{H}_n = \{ h\colon D_h > \vep_n u_n \}$ in the original and the erased model, and $\mathcal{H}_n = \{ h\colon D_h(n) > \vep_n n^\alpha \}$ in the conditioned model.
\end{definition}

Note that $D_h$ refers to the degree before erasure in the erased model and is hence the same as in the original model. Vertices in $\mathcal{H}_n^c=[n]\setminus \mathcal{H}_n$ are referred to as {\em normal} vertices. An important consequence of the definition of giant-degree vertices is that other vertices are connected solely to them:

\begin{lemma}[Neighbors are giants]
\label{only_to_giants}
Assume that \eqref{svfunction} holds. A uniformly chosen vertex is with high probability only connected to giant-degree vertices.
\end{lemma}

\begin{remark} \label{rem:only_to_giants} The lemma is essentially established by showing that a randomly chosen half-edge is with high probability connected to a giant-degree vertex. This observation will be used to establish Corollary \ref{thm:multiple}.
\end{remark}

\begin{proof}
The lemma is a consequence of the fact that the total degree of normal vertices is negligible compared to the total degree of the giants. This is proved in \cite[Lemma 2.2 and Lemma 3.1]{vdEvdHHZ05} for the original and the erased model, respectively, but since the definitions of giant-degree vertices differ slightly from ours we give a brief sketch here.

Let $D_{\sss (1)}\leq D_{\sss (2)}\leq \ch{\cdots \leq} D_{\sss (n)}$ denote the order statistics of the degrees, so that $D_{\sss (1)}$ is the smallest degree, $D_{\sss (2)}$ the second smallest and so on. It follows from \cite[Lemma 2.1]{vdEvdHHZ05} that $\P(D_{\sss (N-k)}\geq \vep_nu_n)\to 1$ for any $k$, and hence with high probability
	$$
	\frac{K_n}{L_n}\leq 1-\frac{\sum_{i=0}^{k-1}D_{\sss (N-i)}}{L_n}.
	$$
Also by \cite[Lemma 2.1]{vdEvdHHZ05},
	$$
	\frac{\sum_{i=0}^{k-1}D_{\sss (N-i)}}{L_n}\stackrel{d}{\to}\frac{\sum_{i=1}^k\xi_i}{\sum_{i=1}^\infty\xi_i},
	$$
where $\ch{(\xi_i)_{i\geq 1}}$ are independent almost surely finite \ch{and non-negative} random variables \ch{such that $\sum_{i=1}^\infty\xi_i<\infty$.} Since this holds for any $k$, we conclude that $K_n/L_n\to 0$ in probability so that half-edges of a randomly chosen vertex in the original model are with high probability connected to half-edges of giant-degree vertices. Since giants in the erased model are defined based on the non-erased degree, we draw the same conclusion there.



\ch{To deal with the conditioned model, with degree distribution given by \eqref{cond_degrees}, write $M_n$ for the total degree in this case. Then
	\eqn{
	\E[M_n]=\frac{n}{F(n^\alpha)}\sum_{j=1}^{n^\alpha}[\P(D>j)-\P(D>n^\alpha)].
	}
Recall \eqref{svfunction}. By Karamata's Theorem \cite[Theorems 1.7.2 and 2.6.1]{BinGolTeu89},
	\eqn{
	\sum_{j=1}^{n^\alpha}\P(D>j)=(1+o(1)) \ell(n^\alpha) (2-\tau)^{-1} n^{\alpha(2-\tau)},
	}
while $n^\alpha\P(D>n^\alpha)=\ell(n^\alpha)n^{\alpha(2-\tau)}$. Thus,
	\eqn{
	\E[M_n]=(1+o(1)) \frac{\tau-1}{2-\tau} \ell(n^\alpha)n^{1+\alpha(2-\tau)}.
	}
Further, 
	\eqan{
	{\rm Var}(M_n)&\leq n\expec[D(n)^2]\leq  n\cdot n^\alpha\E[D(n)]= n^\alpha \E[M_n]\ll \E[M_n]^2,
	}
since $\alpha<1+\alpha(2-\tau)$. In particular, $M_n/\E[M_n]\convp 1.$
Similarly, with $K_n^\alpha$ denoting the total degree of non-giants, 
	\eqn{
	\E[K_n^\alpha]\leq \frac{n}{F(n^\alpha)}\sum_{j=1}^{\vep_n n^\alpha}\P(D>j)=(1+o(1)) \ell(\vep_n n^\alpha) (2-\tau)^{-1} (\vep_n n)^{\alpha(2-\tau)}=o(\E[M_n]),
	}
since $\alpha<1/(\tau-1)$ and since $\ell(\vep_n n^\alpha)/\ell(n^\alpha)\leq c \vep_n^{-\delta}$ for any $\delta>0$ by Potter's Theorem \cite[Theorem 1.5.6]{BinGolTeu89}. Thus,$K^\alpha_n/M_n\convp 0$.}
\end{proof}

For two given vertices $u$ and $v$, write $T_i(u,v)$ for the first passage time between $u$ and $v$ in a one-type FPP process based on $\ch{(X_i(e))_e}$, that is,
	\begin{equation}\label{fp}
	T_i(u,v)=\inf\left\{\sum_{e\in\Gamma}X_i(e)\colon \Gamma \mbox{ is a path between }u\mbox{ and }v\right\}.
	\end{equation}
The following key result states that the first passage time between two giants is vanishing:

\begin{proposition}[The infection spreads quickly between giants]\label{prop:quickly}
Consider a configuration graph obtained from the original, the conditioned or the erased model and let $h_i$ be a randomly chosen neighbor of vertex $i$ ($i=1,2$). If \eqref{svfunction} and \eqref{passtime_assumption} hold, then, for any $\ep>0$,
	$$
	\lim_{n\to\infty}\P(T_i(h_1,h_2)\geq \ep)= 0, \qquad i=1,2.
	$$
\end{proposition}

\begin{remark} \label{rem:quickly} By Lemma \ref{only_to_giants}, with high probability $h_1$ and $h_2$ are indeed giants. The conclusion of Proposition \ref{prop:quickly} is in fact valid for a wide choice of giant-degree vertices, but we formulate it for neighbors of the vertices 1 and 2 since this is what we will apply it to. In establishing Corollary \ref{thm:multiple}, we will instead of $h_1$ apply it to the vertex $h_1'$ to which the edge with the smallest passage time among all edges incident to the $k_1$ initial type 1 vertices is attached. By Remark \ref{rem:only_to_giants}, the vertex $h_1'$ also has a giant degree.
\end{remark}

\begin{proof}
We prove the claim separately for the three versions of the model. By Lemma \ref{only_to_giants}, both $h_1$ and $h_2$ are with high probability giant-degree vertices, and we will throughout assume that their degrees are at least $\vep_n u_n$.

Consider first the original configuration model and let $E(h_1,h_2)$ denote the number of edges between $h_1$ and $h_2$. Since there are at least $\vep_n u_n$ half-edges attached to each one of $h_1$ and $h_2$, the variable $E(h_1,h_2)$ is stochastically larger than a binomial variable with parameters $\vep_nu_n/2$ and $\vep_n u_n/2L_n$. Indeed, when we go through the first half of the half-edges of $h_1$ and check if they are connected to a half-edge of $h_2$, the outcome is stochastically larger than a sequence of i.i.d.\ trials with the specified success probability, since there are then still at least $\vep_nu_n/2$ half-edges left to connect to at $h_2$ while the total number of available half-edges is at most $L_n$. It follows from \cite[Lemma 2.1]{vdEvdHHZ05} that $\P(L_n<u_n^{1+\delta})\to 1$ as $n\to \infty$ for any $\delta>0$. Let $Y_n$ denote a binomial random variable with parameters $\vep_nu_n/2$ and $\vep_nu_n/\ch{(2u_n^{1+\delta})}$. On the event $\{L_n<u_n^{1+\delta}\}$, we then have that $E(h_1,h_2)$ is stochastically larger than $Y_n$. Now recall Janson's inequality for a binomial random variable $Y_n$ stating that
\begin{equation}\label{Janson}
\P(|Y_n-\E[Y_n]| \geq t)\leq 2 \exp \left\{-\frac{t^2}{2(\E[Y_n]+t/3)}\right\}\quad\mbox{for any }t \geq 0;
\end{equation}
see \cite[Theorem 1]{J02}. Defining
	$$
	f(n)=\frac{\E[Y_n]}{2}=\frac{1}{8}\vep_n^2u_n^{1-\delta},
	$$
it follows from Janson's inequality with $t=f(n)$ that
	$$
	\P(Y_n \leq f(n))\leq 2\exp\{-C\vep_n^2u_n^{1-\delta}\},
	$$
where $C>0$ is a constant. By recalling \eqref{u_n} and picking $\delta$ small so that $(1-\delta)/(\tau-1)>0$, we conclude that $\P(E(h_1,h_2) \leq f(n))\to 0$ as $n\to\infty$, with $f(n)\to\infty$. The passage time between $h_1$ and $h_2$ is smaller than the minimum of the edge passage times of the direct edges between them, implying that
	$$
	\P(T_i(h_1,h_2) \geq \ep)\leq \P\left(\min_{\ch{j\in[f(n)]}}X_i(j) \geq \ep\right)+\P(E(h_1,h_2) \leq f(n)), \qquad i=1,2.
	$$
The assumption \eqref{passtime_assumption} guarantees that the first term on the right hand side goes to 0, which completes the proof for the original model.

Moving on to the erased model, we write $N(u,v)$ for the number of joint neighbors of the vertices $u$ and $v$. It is shown in \cite[Lemma 6.7]{BvdHH10b} that $N(h_1,h_2)/n\stackrel{d}{\to}Y$ as $n\to\infty$, where $h_1$ and $h_2$ are giant-degree vertices and $Y$ a proper random variable. It follows that $\P(N(h_1,h_2) \leq n^\gamma)\to 0$ for any $\gamma<1$. Giant vertices are hence connected to each other by a large number of (disjoint) two-step paths. For $i=1,2$, let $X^{\sss(2)}_i(j)\stackrel{d}{=}X_i(e)+X_i(\tilde{e})$ denote the total passage time of the $j$-th such path. Then,
	$$
	\P(T_i(h_1,h_2) \geq \ep)\leq \P\left(\min_{\ch{j\in[n^\gamma]}}X^{\sss(2)}_i(j) \geq \ep\right) + \P(N(h_1,h_2) \leq n^\gamma), \qquad i=1,2.
	$$
Here $\ch{(X^{\sss(2)}_i(j))_{j\geq 1}}$ are i.i.d.\ and inherit the property that inf$\{\mbox{Supp}(X^{\sss(2)}_i(j))\}=0$ from their summands, implying that the first term converges to 0 for any $\gamma>0$. This proves the claim.

Finally we consider the model where the degrees are conditioned to be at most $n^\alpha$. For $\alpha>1/\tau$ the graph still has the same topology in the sense that the giants constitute a tightly connected complete graph (although the number of giant-degree vertices converges to infinity when $\alpha<1/(\tau-1)$). Let $E^\alpha(h_1,h_2)$ denote the number of edges connecting the two giants $h_1$ and $h_2$. In the same way as when dealing with the original model, the number $E^\alpha(h_1,h_2)$ can be stochastically bounded from below by a binomial random variable $Y^\alpha_n$ with parameters $\vep_nn^\alpha/2$ and $\vep_nn^\alpha/\ch{(2n^{1+\alpha(2-\tau)+\delta})}$, and mean
$
\E[Y^\alpha_n]=\vep_n^2n^{\alpha\tau-1-\delta}/4.
$
Picking $\delta\in(0,\alpha\tau-1)$, the same argument as for the original model yields that $\P(T_i(h_1,h_2) \geq \ep)\to 0$.

For $\alpha<1/\tau$, we have to work slightly harder. Fix $k\in\mathbb{N}$. It has been shown in \cite[Lemma 3.3]{vdEvdHHZ05} that, when $\alpha\in(1/(\tau+k),1/(\tau+k-1))$, the graph distance between any two giant-degree vertices is with high probability at most $k+1$. We claim that in fact there is a {\em large number} of disjoint paths of length at most $2(k+1)$ between two giants. To see this, first note that the number of giant-degree vertices is given by $|\mathcal{H}_n|=\sum_{v}\mathbf{1}_{\{D_v(n)>\vep_n n^\alpha\}}$, with
	\begin{eqnarray*}
	\E[|\mathcal{H}_n|] & = & n \mathbb{P}(D(n) > \vep_n n^{\alpha}) = n \frac{F(n^{\alpha}) - F(\vep_n n^{\alpha})}{F(n^{\alpha})}\geq n(F(n^{\alpha}) - F(\vep_n n^{\alpha})) \\
	&= & \ell(n)n^{1-\alpha(\tau-1)}.
	\end{eqnarray*}
Since $\alpha<1/(\tau-1)$, we have that $1-\alpha(\tau-1)>0$, and an application of Janson's inequality \eqref{Janson} yields that $|\mathcal{H}_n|\geq n^\gamma$ with high probability for $\gamma\in(0,1-\alpha(\tau-1))$. The number of giant-degree vertices hence grows to infinity with $n$. With this observation at hand, we proceed to construct a growing number of paths between $h_1$ and $h_2$ that are with high probability disjoint.

\vspace{0.2cm}
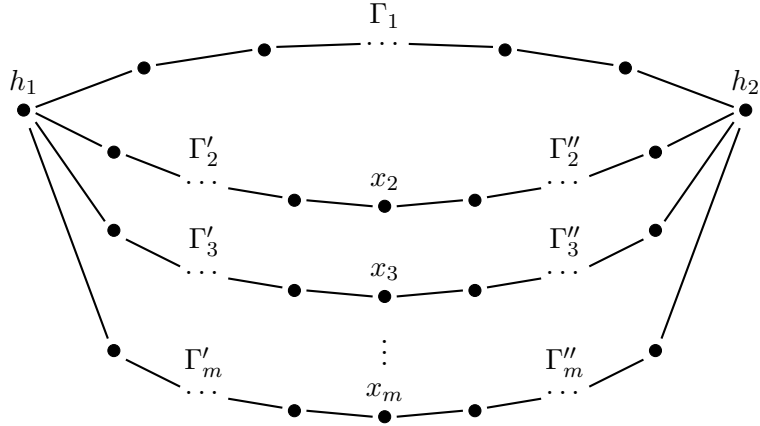
\begin{figure}[htbp]
\begin{center}
\begin{tikzpicture}[scale = 0.8]

\draw[fill] (-6,7) circle (0.1);
\draw[thick] (-5.8,7.05) -- (-4.2,7.65);
\draw[fill] (-4,7.7) circle (0.1);
\draw[thick] (-3.8,7.75) -- (-2.2,8);
\draw[fill] (-2,8) circle (0.1);
\draw[thick] (-1.8,8.05) -- (-0.4,8.1);
\node at (0,8.1) {$\dots$};
\draw[thick] (0.4,8.1) -- (1.8,8.05) ;
\draw[fill] (2,8) circle (0.1);
\draw[thick] (2.2,8) -- (3.8,7.75);
\draw[fill] (4,7.7) circle (0.1);
\draw[thick] (4.2,7.65) -- (5.8,7.05);
\draw[fill] (6,7) circle (0.1);

\draw[thick] (-5.8,6.95) -- (-4.7,6.4);
\draw[fill] (-4.5,6.3) circle (0.1);
\draw[thick] (-4.3,6.25) -- (-3.4,5.9);
\node at (-3,5.8) {$\dots$};
\draw[thick] (-2.6,5.7) -- (-1.7,5.55);
\draw[fill] (-1.5,5.5) circle (0.1);
\draw[thick] (-1.3,5.5) -- (-0.2,5.4);
\draw[fill] (0,5.4) circle (0.1);
\draw[thick] (0.2,5.4) -- (1.3,5.5);
\draw[fill] (1.5,5.5) circle (0.1);
\draw[thick] (1.7,5.55) -- (2.6,5.7);
\node at (3,5.8) {$\dots$};
\draw[thick] (3.4,5.9) -- (4.3,6.25);
\draw[fill] (4.5,6.3) circle (0.1);
\draw[thick] (4.7,6.4) -- (5.8,6.95);

\draw[thick] (-5.8,6.8) -- (-4.65,5.15);
\draw[fill] (-4.5,5) circle (0.1);
\draw[thick] (-4.3,4.85) -- (-3.4,4.4);
\node at (-3,4.3) {$\dots$};
\draw[thick] (-2.6,4.2) -- (-1.7,4.05);
\draw[fill] (-1.5,4) circle (0.1);
\draw[thick] (-1.3,4) -- (-0.2,3.9);
\draw[fill] (0,3.9) circle (0.1);
\draw[thick] (0.2,3.9) -- (1.3,4);
\draw[fill] (1.5,4) circle (0.1);
\draw[thick] (1.7,4.05) -- (2.6,4.2);
\node at (3,4.3) {$\dots$};
\draw[thick] (3.4,4.4) -- (4.3,4.85);
\draw[fill] (4.5,5) circle (0.1);
\draw[thick] (4.65,5.15) -- (5.8,6.8);

\node at (0,3.1) {$\vdots$};

\draw[thick] (-5.9,6.7) -- (-4.6,3.2);
\draw[fill] (-4.5,3) circle (0.1);
\draw[thick] (-4.3,2.85) -- (-3.4,2.4);
\node at (-3,2.3) {$\dots$};
\draw[thick] (-2.6,2.2) -- (-1.7,2.05);
\draw[fill] (-1.5,2) circle (0.1);
\draw[thick] (-1.3,2) -- (-0.2,1.9);
\draw[fill] (0,1.9) circle (0.1);
\draw[thick] (0.2,1.9) -- (1.3,2);
\draw[fill] (1.5,2) circle (0.1);
\draw[thick] (1.7,2.05) -- (2.6,2.2);
\node at (3,2.3) {$\dots$};
\draw[thick] (3.4,2.4) -- (4.3,2.85);
\draw[fill] (4.5,3) circle (0.1);
\draw[thick] (4.6,3.2) -- (5.9,6.7);

\node [above] at (-6,7.1) {$h_1$};
\node [above] at (6,7.1) {$h_2$};
\node [above] at (0,5.5) {$x_2$};
\node [above] at (0,4) {$x_3$};
\node [above] at (0,2) {$x_m$};
\node [above] at (0,8.2) {$\Gamma_1$};
\node [above] at (-3,5.9) {$\Gamma_2'$};
\node [above] at (3,5.9) {$\Gamma_2''$};
\node [above] at (-3,4.4) {$\Gamma_3'$};
\node [above] at (3,4.4) {$\Gamma_3''$};
\node [above] at (-3,2.4) {$\Gamma_m'$};
\node [above] at (3,2.4) {$\Gamma_m''$};

\end{tikzpicture}
\end{center}
\caption{{\small Construction of the $m$ disjoint paths of length at most $2(k+1)$ from $h_1$ to $h_2$, where $m=m(n)$ grows to infinity with $n$.}}
\label{fig:paths}
\end{figure}


Let $\Gamma_1$ be a path from $h_1$ to $h_2$ of length at most $k+1$. Pick a giant vertex $x_2 \notin \Gamma_1$ -- this is possible with high probability since the number of giant-degree vertices grows to infinity with $n$. Then there exist paths $\Gamma_2'$ and $\Gamma_2''$ of length at most $k+1$ connecting $x_2$ to $h_1$ and $h_2$, respectively. Let $\Gamma_2 = \Gamma_2' \cup \Gamma_2''$, where loops arising from common vertices in $\Gamma_2'$ and $\Gamma_2''$ are removed. Then $\Gamma_2$ constitutes a path between $h_1$ and $h_2$ of length at most $2(k+1)$. We claim that, with high probability as $n \to \infty$, the paths $\Gamma_1$ and $\Gamma_2$ are disjoint, except for the first and last vertices $h_1$ and $h_2$. Let $B_n=\{M_n>\ell(n)n^{1+\alpha(2-\tau)}\}$ and recall from the proof of Lemma~\ref{only_to_giants} that $\P(B_n)\to 1$. Conditionally on all degrees, the probability that there is an edge between two vertices $u$ and $v$ is at most $D_u(n)D_{v}(n)/M_n$. Using the fact that the degrees are at most $n^\alpha$, and letting $u \leftrightarrow v$ denote the event that there is an edge between $u$ and $v$, we obtain that
	$$
	\mathbb{P}(\{u \leftrightarrow v\}\cap B_n) \leq \ell(n)\frac{n^{2\alpha}}{n^{1+\alpha(2-\tau)}} = \ell (n)n^{\alpha \tau -1}.
	$$
Note that, if $\Gamma_1$ and $\Gamma_2$ are not disjoint, then there must exist a vertex in $\Gamma_1$ that is connected to a vertex in $\Gamma_2$. With $A_2$ denoting the event that $\Gamma_1$ and $\Gamma_2$ are disjoint, we hence have that
	\begin{equation}
	\mathbb{P}(A_2^c\cap B_n) \leq 2(k+1)^2 \ell(n)n^{\alpha \tau -1}.
	\end{equation}
Next, pick a giant vertex $x_3 \notin \Gamma_1 \cup \Gamma_2$ connected to $h_1$ and $h_2$ by paths $\Gamma_3'$ and $\Gamma_3''$, respectively, of length at most $\ch{k+1}$. Let $\Gamma_3 = \Gamma_3' \cup \Gamma_3''$, again removing any loops. Then $\Gamma_3$ is a path from $h_1$ to $h_2$ of length at most $2(k+1)$. Denote by $A_3$ the event that $\Gamma_3$ and $\Gamma_1\cup \Gamma_2$ are disjoint. Since $|\Gamma_1\cup \Gamma_2|\leq 3(k+1)$, we obtain as above that
	$$
	\mathbb{P}(A_3^c\cap B_n) = \mathbb{P}(\{\Gamma_3 \cap (\Gamma_1 \cup \Gamma_2) \neq \emptyset\}\cap B_n) \leq 6(k+1)^2 \ell(n)n^{\alpha \tau -1}.
	$$
Iterating this construction, in step $m$ we pick a giant vertex $x_m$ connected to $h_1$ and $h_2$ by paths $\Gamma_m'$ and $\Gamma_m''$, respectively, of length at most $k+1$, and set $\Gamma_m=\Gamma_m'\cup \Gamma_m''$ with loops removed. Define
	$$
	A_m=\left\{\Gamma_m \mbox{ and }\Gamma_1\cup\cdots\cup\Gamma_{m-1} \mbox{ are disjoint} \right\}.
	$$
Then, since $|\Gamma_m|\leq 2(k+1)$ and $|\Gamma_1\cup\ldots\cup\Gamma_{m-1}|\leq (2m-3)(k+1)$, we can bound
	$$
	\mathbb{P}(A_m^c\cap B_n) \leq 2(2m-3)(k+1)^2\ell(n) n^{\alpha \tau -1}.
	$$
\ch{Let} $\bar{A}_m=\cap_{j=2}^mA_j$ denote the event that there is no overlap between any of the paths that we have constructed. Then
	\begin{eqnarray*}
	\P(\bar{A}_m\cap B_n) & = & \P\bigg(\bigcup_{j=2}^m A_j^c\cap B_n\bigg) \leq  \sum_{j=2}^m \mathbb{P}(A_j^c\cap B_n) \\
	& \leq & \, 2(k+1)^2 \ell(n)n^{\alpha \tau -1}\sum_{j=2}^m (2j-3) \\
	&\leq & \, 2(k+1)^2 m^2\ell(n)n^{\alpha \tau -1}.
	\end{eqnarray*}
Since $\alpha<1/\tau$, we have that $\alpha\tau-1>0$. Pick $\delta\in(0,\alpha\tau-1)$ and take $m=m(n)=n^\delta$ so that $\P(\bar{A}_{m(n)})\to 0$.

\ch{To complete the proof, for} $i=1,2$, let $X_i^{\sss 2(k+1)}$ be a random variable with $X_i^{\sss 2(k+1)}\stackrel{d}{=}X_i(1)+\ldots +X_i(2(k+1))$, that is, $X_i^{\sss 2(k+1)}$ is distributed as the passage time of a given path of length $2(k+1)$. On the event $\bar{A}_m$ that the $m$ paths that we have constructed between $h_1$ and $h_2$ are disjoint, the passage time $T_i(h_1,h_2)$ between $h_1$ and $h_2$ is bounded from above by the minimum of $m$ i.i.d.\ copies of $X_i^{\sss 2(k+1)}$. Under the assumption \eqref{passtime_assumption}, such a minimum converges to 0 in probability, since $m(n)=n^\delta\to\infty$. Let $\{X_i^{\sss(2(k+1))}(j)\}_{j\geq 1}$ be the i.i.d.\ sequence. The proof is completed by noting that
	$$
	\P(T_i(h_1,h_2) \geq \ep)\leq \P\left(\min_{\ch{j\in[ m(n)]}}X_i^{\sss(2(k+1))}(j) \geq \ep\right)+\P(\bar{A}_m\cap B_n)+\P(B_n), \qquad i=1,2,
	$$
where all terms have been shown above to converge to 0.
\end{proof}

\begin{remark}[Avoiding the initial vertices]
\label{rem}
Note from the proof that Proposition \ref{prop:quickly} remains true also when vertex 1 and 2 (or any finite set of vertices) are not allowed to be used to transfer the infection from $h_1$ to $h_2$, that is, when the infimum in \eqref{fp} is taken over all paths $\Gamma$ that do not contain vertex 1 or 2. This will be relevant when applying the result to the competition model below.
\end{remark}

With Lemma~\ref{only_to_giants} and Proposition~\ref{prop:quickly} in hand, we are ready to prove Theorems~\ref{thm:o}, \ref{thm:c} and \ref{thm:e}. Given that Lemma~\ref{only_to_giants} and Proposition~\ref{prop:quickly} apply for all versions of the model, the proofs are identical and can be merged.

\begin{proof}[Proof of Theorems~\ref{thm:o}, \ref{thm:c}, \ref{thm:e}]
We show that, for any $\delta>0$, there exists $n_\delta$ such that
\begin{equation}\label{delta}
\P(\ch{N_2(n)>1}\mid Z_1<Z_2)\leq \delta \quad\mbox{when }n\geq n_\delta,
\end{equation}
that is, $\P(N_2(n)=1\mid Z_1<Z_2)\to 1$. That $\P(N_1(n)=1\mid Z_1>Z_2)\to 1$ is proved analogously.

Let $\cN_i$ denote the set of neighbors of vertex $i$ and write $h^*$ for the first vertex in $\cN_1\cup \cN_2$ that is infected. By definition, if $Z_1<Z_2$, then $h^*=h_1\in\cN_1$. Write $\widetilde{T}_i(h_1,h)$ for the type $i$ first passage time between $\ch{h_1}$ and $h$ when vertex 1 and 2 are not allowed to be used and note that, for any $\ep>0$, conditionally on $Z_1<Z_2$,
$$
\{N_2(n)>1\}\subset\left\{\max_{h\in\cN_2}\widetilde{T}_1(h_1,h)\geq \ep\right\}\cup\{Z_2-Z_1\leq\ep\}.
$$
Indeed, if all vertices in $\cN_2$ are reached by type 1 from $h_1$ within time $\ep$, while the time that it takes for type 2 to reach any vertex in $\cN_2$ is larger than $\ep$, then type 2 \ch{is not} able to make any progress at all, and \ch{thus ends} up occupying only its initial site.

Now pick $\ep$ small so that $\P(Z_2-Z_1\leq \ep\mid Z_1<Z_2)<\delta/2$. This is possible since $Z_2-Z_1$ is a proper random variable with support on $(0,\infty)$. Also fix $d$ such that $\P(|\cN_2|>d\mid Z_1<Z_2)<\delta/4$ and observe that
	$$
	\P\left(\max_{h\in\cN_2}\widetilde{T}_1(h_1,h)\geq \ep\mid Z_1<Z_2\right)\leq \P\left(\max_{h\in\cN_2\colon |\cN_2|\leq d}\widetilde{T}_1(h_1,h)\geq \ep\mid Z_1<Z_2\right)+\delta/4.
	$$
The event $Z_1<Z_2$ only contains information about vertex 1 and 2 and hence does not affect $\widetilde{T}_1(h_1,h)$. Combining this with a union bound, we obtain that
	$$
	\P\left(\max_{h\in\cN_2\colon |\cN_2|\leq d}\widetilde{T}_1(h_1,h)\geq \ep\cap G_n \mid Z_1<Z_2\right)\leq d\cdot\P(\ch{\widetilde{T}_1(h_1,h_2)\geq \ep}),
	$$
where $h_1$ and $h_2$ are randomly chosen neighbors of vertex 1 and 2, respectively. Here, the right hand side converges to 0 by Proposition~\ref{prop:quickly} and Remark \ref{rem} and so can be made smaller than $\delta/4$ by picking $n$ large. This concludes the proof of \eqref{delta}.
\end{proof}

Finally, we indicate how the above proof can be generalized to establish Corollary~\ref{thm:multiple} and its counterpart for the erased and conditioned model:

\begin{proof}[Proof of Corollary~\ref{thm:multiple}]
Generalizing \eqref{delta}, we now need to show that, for any $\delta>0$, there exists $n_\delta$ such that
	\begin{equation}\label{deltageneralized}
	\P(\ch{N_2(n)>k_2}\mid Z_{1,k_1}<Z_{2,k_2})\leq \delta \quad\mbox{when }n\geq n_\delta.
	\end{equation}
To this end, consider all edges attached to the $k_1$ initial type 1 vertices and let $h_1'$ be the other end point of the edge with the smallest passage time -- this is the vertex that is infected when type 1 makes its first move. Note that \eqref{deltageneralized} can be proved in a similar way as above by showing that, for any $\ep>0$, conditionally on $Z_{1,k_1}<Z_{2,k_2}$,
$$
\{N_2(n)>k_2\}\subset\left\{\max_{h\in\cN_2}\widetilde{T}_1(h_1',h)\geq \ep\right\}\cup\{Z_{2,k_2}-Z_{1,k_1}\leq\ep\},
$$
where both events on the right hand side have probability smaller than $\delta/2$. Here, to bound the first term, we use Remark \ref{rem:quickly}. We conclude that $\P(N_2(n)=k_2 \mid Z_{1,k_1}<Z_{2,k_2})\to 1$. If $k_2$ is fixed while $k_1=k_1(n)\to\infty$, then $\P(Z_{1,k_1}<Z_{2,k_2})\to 1$, implying that $\P(N_2(n)=k_2)\to 1$.
\end{proof}

\section{Further work}
\label{further}

There are a number of aspects of the model treated here that \ch{deserve} further attention. We give a few examples below.\medskip

\noindent\textbf{Growing initial sets.} When {\em both} initial sets grow with $n$, we believe that the outcome of the competition depends on the combination of growth rates of the initial sets and the passage time distribution. For a growing initial set, the time until the corresponding infection type reaches a giant vertex should converge to 0 in probability, with the rate of convergence being determined by the growth of the set, \ch{as well as the behavior close to zero of the passage-time distribution}. \ch{Also} the rate at which the first passage time between two giants converges to 0 depends on the passage-time distribution.
\medskip

\noindent\textbf{Other passage-time distributions.} What happens for continuous passage-time distributions that do not fulfil our assumption \eqref{passtime_assumption}? We believe that the outcome of the competition may then depend not only on the infimum of the supports, but also on other properties of the distributions. In general, if the passage-time distribution does not have support down to 0, then the first passage time between two giants for the corresponding type is not vanishing, implying that the other type is not cut off from the possibility of capturing vertices beyond its initial set. However, a type that does not have support down to 0 may still have the possibility of occupying {\em all} initially uninfected vertices: Suppose that type 1 has a passage time with support down to 0, while type 2 has a passage time with support down to $\eta>0$. Suppose also that the support of \ch{the} passage time of type 1 is larger than $2\eta$. Then, with positive probability, $Z_1>2(\eta+\ep)$ for some $\ep>0$, while $Z_2\leq \eta+\ep$. \ch{As a result}, from the giant neighbor that type 2 reaches at time $Z_2$, \ch{with high probability conditionally on $Z_1>2(\eta+\ep)$ and $Z_2\leq 2(\eta+\ep)$,} {\em all} other giants \ch{are} reached at time $2(\eta+\ep)$ \ch{by type 2}, since the giant is connected to all other giants with an increasing number of edges so the passage time \ch{of type 2 is} close to $\eta$. Since vertex 1 is only connected to giants, at time $2(\eta+\ep)$ \ch{it becomes} isolated, even though the support of its passage times \ch{does go} down to zero, while the one for type 2 does not. \ch{We conclude that type 2 wins with positive probability, but naturally also type 1 wins with positive probability. Working out what the exact winning probabilities are for each of the two types in cases where the supports do not go down to zero is quite interesting.}

Another option that might seem natural is to consider {\em discrete} passage times. In order to be meaningful, however, this would presumably require some type of non-lattice condition guaranteeing that both types cannot arrive at a vertex at the same time. If both types arrive simultaneously at a vertex, then a \ch{tie-breaking rule} is needed to decide the type that occupies the vertex and due to the special structure of the graph for $\tau\in(1,2)$ it is likely that the choice of tie-breaker will in fact decide the competition. In less heavy-tailed regimes, the tie-breaker \ch{typically kicks} in only when the competition is already decided and \ch{is} thereby essentially irrelevant; see e.g.\ \cite{BvdHK15,vdHK15}.\medskip

\noindent\textbf{Conditioned model with stricter upper bound.} In the conditioned model, the degrees are conditioned to be smaller than $n^\alpha$ for some $\alpha>0$. The upper bound could also be taken as a more general function $a_n$ of $n$. When $a_n$ grows more slowly with $n$, the conditioning \ch{has} a larger impact on the structure of the graph. Specifically, when $a_n$ is sufficiently small, the graph may loose the property that any two vertices of maximal degree are within finite distance from each other. This would be interesting to study in its own right, but would also have implications for the competition process in that the first passage time between two giants is then no longer vanishing. We believe that, for a certain class of \ch{natural passage-time distributions including the exponential and uniform distributions}, the cut-off may be at $a_n\propto \log n$. \ch{For other distributions having a thicker or thinner tail close to 0, we believe the cut-off to be at a different value of $a_n$. We} leave this for future research. 

\end{document}